\documentclass[11pt]{article}

\usepackage{makeidx}
\makeindex
\usepackage{amsmath}
\usepackage{amsfonts}
\usepackage{amssymb}
\usepackage{amsthm}
\usepackage{MnSymbol}
\usepackage{relsize}
\usepackage{etoolbox}
\theoremstyle{plain}
\newtheorem{theorem}{Theorem}[section]

\newtheorem{lemma}[theorem]{Lemma}
\newtheorem{proposition}[theorem]{Proposition}
\newtheorem{corollary}[theorem]{Corollary}
\theoremstyle{definition}

\theoremstyle{remark}
\newtheorem*{remark}{Remark}

\usepackage[margin=1.25in]{geometry}

\begin{document}

\title{Spectral Bounds for Polydiagonal Jacobi Matrix Operators}
\author{Arman Sahovic}
\date{\today}
\maketitle

\begin{abstract}

The research on spectral inequalities for discrete Schr\"odinger Operators has proved fruitful in the last decade. Indeed, several authors analysed the operator's canonical relation to a tridiagonal Jacobi matrix operator. In this paper, we consider a generalisation of this relation with regards to connecting higher order Schr\"odinger-type operators with symmetric matrix operators with arbitrarily many non-zero diagonals above and below the main diagonal. We thus obtain spectral bounds for such matrices, similar in nature to the Lieb--Thirring inequalities.

\end{abstract}

\section{Background}

Let $W$ be the self-adjoint Jacobi matrix operator acting on $\ell^2(\Bbb Z)$:

\begin{equation*}
W = \left(
\begin{array}{cccccc}
\ddots & \vdots & \vdots & \vdots & \vdots & \udots\\
\ldots & b_{-1} & a_{-1} & 0   & 0   & \ldots \\
\ldots & a_{-1} & b_0 & a_0 & 0   & \ldots \\
\ldots & 0   & a_0 & b_1 & a_1   & \ldots \\
\ldots & 0   & 0   & a_1 & b_2 & \ldots \\
\udots & \vdots & \vdots & \vdots & \vdots & \ddots
\end{array} \right),
\end{equation*}
\vspace{0.2cm}
\noindent via:
\begin{equation*}
(W\varphi)(n)=a_{n-1}\varphi(n-1)+b_n\varphi(n)+a_n\varphi(n+1), \quad \textup{for } n\in\mathbb Z,
\end{equation*}

where $a_n> 0$ and $b_n\in\Bbb R$.
This operator can be viewed as the one-dimensional discrete Schr\"odinger operator if $a_n= 1$ for all $n$.
A variety of papers examined such operators, for example, we quote the work by R. Killip and B. Simon in \cite{KS03}, where they obtained sum rules for such Jacobi matrices.
Additionally, D. Hundertmark and B. Simon in \cite{HS02} were able to find spectral bounds for these operators. We thus state their result:

If $a_n\to 1$, $b_n\to 0$ rapidly enough, as $n\to\pm\infty$, the essential  spectrum $\sigma_{ess}(W)$ of W is absolutely continuous and coincides with the interval $[-2,2]$ (see for example \cite{Bak99}). Besides, $W$ may have simple eigenvalues $\{E_j^\pm\}_{j=1}^{N_\pm}$ where $N_\pm\in \overline{\mathbb{N}}(:=\mathbb N \cup \{\infty\})$, and
$$
E_1^{+}>E_2^{+}>...>2>-2>...>E_2^{-}>E_1^{-}.
$$

\noindent Indeed, they found:
\begin{theorem}
If $\{b_n\}_{n\in\Bbb Z}, \, \{a_n-1\}_{n\in\Bbb Z}\in \ell^{\gamma+1/2}(\mathbb{Z})$, $\gamma\geq1/2$, then
\begin{equation}\label{HS2}
\sum_{j=1}^{N_+}|E^+_j-2|^\gamma+\sum_{j=1}^{N_-}|E^-_j+2|^\gamma \leq k_\gamma\left[\sum_{n=-\infty}^\infty |b_n|^{\gamma+1/2}
+4\sum_{n=-\infty}^\infty|a_n-1|^{\gamma+1/2}\right]
\textup{,}
\end{equation}
where
$$
k_\gamma=2(3^{\gamma-1/2})L^{cl}_{\gamma,1},
\qquad
\text{and}
\qquad
L^{cl}_{\gamma,1}=\frac{\Gamma(\gamma+1)}{2\sqrt{\pi}\,\,\Gamma(\gamma+3/2)}.
$$

\end{theorem}

The author (see \cite{Sah10}) then improved their result, achieving the smaller constant:
$k_\gamma=3^{\gamma-1}\,\pi\,L^{cl}_{\gamma,1}$, by translating a well-known method employed by A. Laptev, Dolbeaut and Loss in \cite{DLL08} to the discrete scenario. They, in turn, used a simple argument by A. Eden and Foias (see \cite{EF91}) to obtain improved constants for Lieb-Thirring inequalities in one dimension.

The aim of this paper is to answer the natural question of whether these methods can be generalised to give bounds for higher order Schr\"odinger-type operators and thus 'polydiagonal' Jacobi-type matrix operators, which we shall define below.

\section{Notation and Preliminary Material}

For a sequence $\{\varphi(n)\}_{n\in\Bbb Z}$, let $D$ and $D^*$ be the difference operator and its adjoint respectively, denoted by
$
D {\varphi(n)} = {\varphi(n+1)-\varphi(n)},$ and $ D^* \varphi(n) = \varphi(n-1)-\varphi(n).
$ We then denote the discrete one-dimensional Laplacian by $\Delta_D:=D^*D=-\varphi(n+1)+2\varphi(n)-\varphi(n-1)$. For $\sigma\in\Bbb N,\,\,n\in\Bbb Z$ and a sequence $\varphi\in\ell^2(\Bbb Z)$, $\Delta_D^\sigma$ will be defined by:
$$
(\Delta_D^\sigma\varphi)(n):=(\Delta_D (\Delta_D^{\sigma-1}\varphi))(n).
$$

\noindent We note that $\Delta_D$ being self-adjoint immediately implies that $\Delta_D^\sigma$ is also self-adjoint. 

Finding an explicit formula for $\Delta_D^\sigma$ requires a few combinatorial techniques, all of which are standard:
Let $^{a}C_{b}:=\binom{a}{b}:=a!/((a-b)!b!)$, for $a,\,b\,\in\mathbb{Z}$. Then we have:
(i)
$^aC_b+^aC_{b+1}= {^{a+1}C_{b+1}}$,
(ii)
$2\,{^aC_0}+{^aC_1}={^{a+2}C_1}$,
(iii)
 $2\,{^aC_a}+{^aC_{a-1}}={^{a+2}C_{a+1}}$.

\noindent A simple induction argument then delivers  our formula for the $\sigma^{th}$ order discrete Laplacian operator:
$$
(\Delta_D^\sigma\varphi)(n)=\sum_{k=0}^{2\sigma}\,{^{2\sigma}C_{k}}(-1)^{k+\sigma}\varphi(n-\sigma+k).
$$

\noindent Furthermore, in order to identify our essential spectrum, we apply the discrete fourier transform:
$$
        \mathcal{F}(\Delta_D^\sigma\varphi)(x)  =  \sum_{n\in\mathbb{Z}} e^{inx}\Bigl(\sum_{k=0}^{2\sigma}{^{2\sigma}C_{k}}(-1)^{k+\sigma}\varphi(n-\sigma+k)\Bigr)
$$

\noindent which, after some rearrangement, yields:
$$
\mathcal{F}(\Delta_D^\sigma\varphi)(x)
= 
\Bigl[{^{2\sigma}C_{\sigma}}+2\sum^{\sigma-1}_{k=0} \,\,{^{2\sigma}C_{k}}(-1)^{k+\sigma}\cos((\sigma-k)x)\Bigr](\mathcal{F}\varphi)(x).
$$
The essential spectrum of the operator $\Delta_D^\sigma$ will thus be the range of the above symbol, which can be found to be $\varsigma_{\text{ess}}(\Delta^\sigma_D)=[0,4^\sigma]$.
\section{Main Results}

We now let $\{\psi_j\}_{j=1}^N$, $N\in\overline{\mathbb{N}}$ be the orthonormal system 
of eigensequences in $\ell^2(\Bbb Z)$ corresponding to the negative eigenvalues $\{e_j\}_{j=1}^{N}$
of the $(2\sigma)^{th}$ order discrete Schr\"odinger-type operator:
\begin{equation}\label{N-Schrodinger}
(H^\sigma_D\psi_j)(n):=(\Delta_D^{\sigma}\psi_j)(n) - b_n\psi_j(n) = e_j \psi_j(n),
\end{equation}
where $j\in\{1,\ldots,N\}$ and we assume that $b_n\geq0$ for all $n\in\Bbb{Z}$.
 Our next result is concerned with estimating those negative eigenvalues:

\begin{theorem} \label{L-Tsigma}

Let $b_n\geq 0$, $\{b_n\}_{n\in\Bbb Z} \in \ell^{\gamma+1/2\sigma}(\mathbb Z)$, $\gamma \ge 1$. Then the negative eigenvalues  $\{e_j\}_{j=1}^N$ of the operator $H^\sigma_D$ satisfy the inequality 
\begin{eqnarray*}\label{eq:AisLiebN}     	%
\sum^{N}_{j=1} |e_j|^\gamma   &\leq &     \eta_{\sigma}^{\gamma} \sum_{n\in\Bbb Z} 	b_n^{\gamma + 1/2\sigma},
\end{eqnarray*}
where
$$
\eta_{\sigma}^{\gamma}\,\,:=\,\,\frac{2\sigma}{(2\sigma+1)^{(2\sigma+1)/2\sigma}}\frac{\Gamma(\frac{4\sigma+1}{2\sigma})\Gamma(\gamma+1)}{\Gamma(\gamma+\frac{2\sigma+1}{2\sigma})}.
$$
\end{theorem}
\begin{remark}
As the discrete spectrum of $H^\sigma_D$ lies in $[-\infty,0]$ and $[4^\sigma,\infty]$, we shift our operator to the left by $4^\sigma$ and by analogy have an estimate for the positive eigenvalues of that operator, thus immediately obtaining Corollary \ref{LiebGammaN+}:
\end{remark}


\begin{corollary}\label{LiebGammaN+}  
Let $b_n\geq 0$, $\{b_n\}_{n\in\Bbb Z} \in \ell^{\gamma+1/2\sigma}(\mathbb Z)$, $\gamma \ge 1$. Then the positive eigenvalues  $\{e_j\}_{j=1}^{N}$ of the operator $\Delta_D^{\sigma}-4^\sigma+b$  satisfy the inequality 
$$     	%
\qquad\qquad\qquad\qquad\qquad\sum_{j=1}^{N} e_j^\gamma   \leq       \eta_{\sigma}^{\gamma}\sum_{n\in{\Bbb Z}} 	b_n^{\gamma + 1/2\sigma},\qquad\qquad\qquad\text{with}\,\,\,\eta_\sigma^\gamma\,\,\text{given above}.
$$
\end{corollary}
\bigskip

\noindent Finally we will apply these results to obtain spectral bounds for
the following operator:

\noindent We let $W_{\sigma}$ be a polydiagonal self-adjoint Jacobi-type matrix operator: 

\begin{equation*}
W_\sigma := \small\left(
\begin{array}{cccccccccc}
\ddots & \ddots & \ddots & \ddots & \ddots  & \ddots &\ddots &\ddots & \udots\\
\ddots & \ddots & \ddots & \ddots & \ddots & a^\sigma_{-2} & 0  &  0 & \ddots \\
\ddots & \ddots & b_{-1} & a^1_{-1} & \ddots & \ddots & a^\sigma_{-1} &0 & \ddots \\
\ddots & \ddots & a^1_{-1} & b_{0} & a^1_0 & \ddots & \ddots &a^\sigma_{0} & \ddots \\
\ddots & \ddots & \ddots & a_0^1  & b_{1} & a^1_1 & \ddots  &\ddots  &  \ddots \\
\ddots & a^\sigma_{-2} & \ddots & \ddots & a^1_{1} & b_2  & a^1_2   &\ddots &  \ddots \\
\ddots & 0 & a^\sigma_{-1} & \ddots & \ddots  & a^1_2  & b_3  &\ddots   &  \ddots \\
\ddots & 0 & 0 & a^\sigma_{0}  & \ddots & \ddots  &  \ddots & \ddots & \ddots \\
\udots & \ddots& \ddots & \ddots & \ddots  & \ddots & \ddots  &  \ddots &\ddots
\end{array} \right),\\
\end{equation*}

\bigskip

\noindent viewed as an operator acting on $\ell^2(\Bbb{Z})$ as follows: For $n\in\mathbb Z,i\in\mathbb \{1,\ldots,\sigma\}$:
\begin{eqnarray*}\label{Wsigma}
(W_{\sigma} \varphi)(n) &=& \sum_{i=1}^{\sigma}a^i_{n-i} \varphi(n-i)+b_n \varphi(n)+\sum_{i=1}^{\sigma}a^i_{n} \varphi(n+i)\\
&=&
a^\sigma_{n-\sigma} \varphi(n-\sigma)+\ldots+
a^1_{n-1} \varphi(n-1)+b_n \varphi(n)+a^1_n \varphi(n+1)+
\ldots+a^\sigma_{n} \varphi(n+\sigma),
\end{eqnarray*}

where $a^i_n$, $b_n\in\mathbb R$, for all $ i\in\mathbb \{1,\ldots,\sigma\}$. We denote $(W_\sigma(\{a_n^1\},\ldots,\{a_n^\sigma\},\{b_n\})\varphi)(n):=(W_\sigma\varphi)(n)$ where we understand $\{.\}$ to mean $\{.\}_{n\in\Bbb Z}$. We are then interested in perturbations of the following special case:
\begin{equation*}
(W_\sigma^0\varphi)(n):=(W_\sigma(\{a_n^1\equiv\omega_1\},\ldots,\{a_n^\sigma\equiv\omega_\sigma\},\{b_n\equiv 0\})\varphi)(n),
\end{equation*}
\noindent where $\omega_i:= {^{2\sigma}C_{\sigma+i}}(-1)^{i}$, and explicitly:
\begin{equation*}
(W_\sigma^0\varphi)(n)=((\Delta_D^\sigma-{^{2\sigma}C_{\sigma}})\varphi)(n)=\sum_{k=0,\,\,k\neq \sigma}^{2\sigma}{^{2\sigma}C_{k}}(-1)^{k+\sigma}\varphi(n-\sigma+k),
\end{equation*}
called the free Jacobi-type matrix of order $\sigma$. In particular, we examine the case where $W_\sigma-W^0_\sigma$ is compact.
Thus in what follows we assume that our sequences tend to the operator coefficients rapidly enough, i.e.,
$a^i_n\to \omega_i$, $b_n\to 0$, as $n\to\pm\infty$. Then the essential spectrum $\varsigma_{ess}$ is given by $\varsigma_{ess}(W_\sigma)= \varsigma_{ess}(W^0_\sigma)= [-{^{2\sigma}C_{\sigma}},4^\sigma-{^{2\sigma}C_{\sigma}}]$ and $W_\sigma$ may have simple eigenvalues $\{E_j^\pm\}_{j=1}^{N_\pm}$ where $N_\pm\in \overline{\mathbb{N}}$, and
$$
E_1^{+}>E_2^{+}>...>4^\sigma-{^{2\sigma}C_{\sigma}}>-{^{2\sigma}C_{\sigma}}>...>E_2^{-}>E_1^{-}.
$$

\begin{theorem}\label{JacobiPoly}
Let $\gamma\geq1,\,\,\{b_n\}_{n\in\Bbb Z}, \, \{a^i_n-\omega_i\}_{n\in\Bbb Z}\in \ell^{\gamma+1/2\sigma}(\mathbb Z)\,$ for all $i\in \{1,\ldots,\sigma\}$.  Then for the eigenvalues $\{E_j^\pm\}_{j=1}^{N_\pm}$ of the operator $W_{\sigma}$ we have: 
\begin{multline*}
\sum_{j=1}^{N_-} |E^{-}_{j}+{^{2\sigma}C_{\sigma}}|^\gamma +
\sum_{j=1}^{N_+} |E^{+}_{j}-(4^\sigma-{^{2\sigma}C_{\sigma}})|^\gamma
\leq
\nu^\gamma_\sigma 
\left(\sum_{n\in\mathbb{Z}} |b_n|^{\gamma+1/2\sigma}+4\sum_{n\in\mathbb{Z}}\sum^\sigma_{k=1}|a^k_{n}-\omega_k|^{\gamma+1/2\sigma}\right),
\end{multline*}
\noindent where
$$
\nu^\gamma_\sigma=2\sigma\,\,(2\sigma+1)^{\gamma-2}\,\,\frac{\Gamma(\frac{4\sigma+1}{2\sigma})\Gamma(\gamma+1)}{\Gamma(\gamma+\frac{2\sigma+1}{2\sigma})} .
$$

\end{theorem}

\section{Auxiliary Results}
We require the following discrete Kolmogorov-type inequality:

\begin{lemma}\label{Ditzian}
For a sequence $\varphi\in\ell^2({\Bbb Z}),\,$ and for $\,\,n>k\geq 1$, we have the following inequality:
$$
\|D^k\varphi \|_{\ell^2(\mathbb{Z})} \leq \|\varphi \|_{\ell^2(\mathbb{Z})}^{1-k/n} \|D^n\varphi \|_{\ell^2(\mathbb{Z})}^{k/n}.
$$
\end{lemma}

\begin{proof}
We proceed by induction and assume we have the required inequality for $n\leq m$. Then:
\begin{eqnarray*}
\| D^m\varphi  \|^2_{\ell^2(\mathbb{Z})}  
= 
\langle D^{m}\varphi , D^{m}\varphi \rangle
= 
\langle D^*D^{m}\varphi , D^{m-1}\varphi \rangle
\leq 
\|D^{m+1}\varphi \|_{\ell^2(\mathbb{Z})}\| D^{m-1}\varphi \|_{\ell^2(\mathbb{Z})},
\end{eqnarray*}

\noindent We thus apply our induction hypothesis, and set $k=m-1$ and $n=m$:
\begin{eqnarray*}
& \| D^m\varphi  \|^2_{\ell^2(\mathbb{Z})} 
& \leq 
\| D^{m+1}\varphi \|_{\ell^2(\mathbb{Z})}\| D^{m}\varphi \|_{\ell^2(\mathbb{Z})}^{(m-1)/m} \|\varphi\|_{\ell^2(\mathbb{Z})}^{1/m}\\
 \Rightarrow &
\| D^m\varphi  \|_{\ell^2(\mathbb{Z})}  
&\leq 
\| D^{m+1}\varphi \|_{\ell^2(\mathbb{Z})}^{m/(m+1)}\|\varphi\|_{\ell^2(\mathbb{Z})}^{1/(m+1)}.
\end{eqnarray*}
We now return to the induction hypothesis:
\begin{eqnarray*}
\|D^k\varphi\|_{\ell^2(\mathbb{Z})}
& \leq &
\|D^m\varphi\|_{\ell^2(\mathbb{Z})}^{k/m}\|\varphi\|_{\ell^2(\mathbb{Z})}^{(m-k)/m}\\
& \leq & 
\| D^{m+1}\varphi \|_{\ell^2(\mathbb{Z})}^{k/(m+1)}\|\varphi\|_{\ell^2(\mathbb{Z})}^{k/m(m+1)}\|\varphi\|_{\ell^2(\mathbb{Z})}^{(m-k)/m}\\
& = & 
\| D^{m+1}\varphi \|_{\ell^2(\mathbb{Z})}^{k/(m+1)}\|\varphi\|_{\ell^2(\mathbb{Z})}^{1- k/(m+1)}.
\end{eqnarray*}
\end{proof}

\noindent We are now equipped to prove an Agmon--Kolmogorov--type inequality:
\begin{proposition}\label{AgmonSigma}
For a sequence $\varphi\, \in \ell^2(\mathbb{Z})$, we have for any $\sigma\in \mathbb{N}$:
        $$
        \|\varphi\|_{\ell^\infty(\Bbb Z)}  \leq  \| \varphi \|^{1-1/2\sigma}_{\ell^2(\mathbb{Z})} \| D^{\sigma}\varphi\|^{1/2\sigma}_{\ell^2(\mathbb{Z})}.
        $$
\end{proposition}

\begin{proof}
First we use Lemma \ref{Ditzian} with $k=1,\,n=\sigma$:
\begin{equation*}
\| D\varphi  \|_{\ell^2(\mathbb{Z})}  
\leq 
\| \varphi \|_{\ell^2(\mathbb{Z})}^{1-\frac{1}{\sigma}}\| D^{\sigma}\varphi  \|_{\ell^2(\mathbb{Z})}^{\frac{1}{\sigma}},
\end{equation*}

\noindent and we apply this estimate to the well-known discrete Agmon inequality (see \cite{Sah10}):
\begin{eqnarray*}
\mid\varphi(n)\mid^2  
&\leq &
\| \varphi \|_{\ell^2(\mathbb{Z})} \| D\varphi  \|_{\ell^2(\mathbb{Z})}\\
& \leq & \| \varphi \|_{\ell^2(\mathbb{Z})} \| \varphi \|_{\ell^2(\mathbb{Z})}^{1-\frac{1}{\sigma}}\| D^{\sigma}\varphi  \|_{\ell^2(\Bbb Z)}^{\frac{1}{\sigma}}\\
& = & 
\|\varphi  \|_{\ell^2(\mathbb{Z})}^{2-\frac{1}{\sigma}}\| D^{\sigma}\varphi  \|_{\ell^2(\mathbb{Z})}^{\frac{1}{\sigma}}.
\end{eqnarray*}

\end{proof}

\begin{proposition}\label{DGSI}
Let $\{\psi_j\}_{j=1}^N$ be an orthonormal system of sequences in $\ell^2(\mathbb Z)$, i.e. $\langle\psi_j, \psi_k\rangle=\delta_{jk}$,
and let $\rho(n) := \sum_{j=1}^N |\psi_j(n)|^2$.
Then
\begin{equation*}
\sum_{n\in\Bbb Z} \rho^{2\sigma+1}(n)\, 
 \le \sum_{j=1}^N \sum_{n\in\Bbb{Z}} |{D^\sigma\psi}_j(n)|^2\, .
\end{equation*}
\end{proposition}

\begin{proof}

\noindent Let $\xi=(\xi_1, \xi_2, \dots, \xi_N)\in {\Bbb C}^N$. 
By Proposition \ref{AgmonSigma}, we have:
\begin{eqnarray*}
\Big|\sum_{j=1}^N\xi_j\psi_j(n)\Big|^2 
& \leq &
 \Big\| \sum_{j=1}^N\xi_j\psi_j\Big\|_{\ell^2(\Bbb Z)}^{(2\sigma-1)/\sigma}\,\,\Big\|D^\sigma\sum_{j=1}^N\xi_j\psi_j\Big\|_{\ell^2(\Bbb Z)}^{1/\sigma}\\
&=&
\Bigl(\sum_{j,k=1}^N\xi_j\bar\xi_k\langle\psi_j,\psi_k\rangle\Bigr)^{(2\sigma-1)/2\sigma}\Bigl(\sum_{j,k=1}^N \xi_j\bar\xi_k \langle D^\sigma\psi_j,D^\sigma\psi_k\rangle\Bigr)^{1/2\sigma}\\
& \le & 
\Bigl(\sum_{j=1}^N |\xi_j|^2\Bigr)^{(2\sigma-1)/2\sigma}\Bigl(\sum_{j,k=1}^N\xi_j\bar\xi_k \langle D^\sigma\psi_j,D^\sigma\psi_k\rangle\Bigr)^{1/2\sigma}.
\end{eqnarray*}

\noindent Let $\xi_j := \overline{\psi_j(n)}$ and as $\rho(n) = \sum_{j=1}^N |\psi_j(n)|^2$:
\begin{eqnarray*}
\rho^2(n)	
& \le & 
\rho^{(2\sigma-1)/2\sigma}(n)\Bigl(\sum_{j,k=1}^N \psi_j(n)\overline{\psi_k(n)}\langle D^\sigma\psi_j,D^\sigma\psi_k\rangle\Bigr)^{1/2\sigma}\\
\Rightarrow \rho^{2\sigma+1}(n)
& \le &	
\sum_{j,k=1}^N \psi_j(n)\overline{\psi_k(n)} \langle D^\sigma\psi_j,D^\sigma\psi_k\rangle\\
\Rightarrow\sum_{n\in\Bbb{Z}} \rho^{2\sigma+1}(n)\,
&
\le
&
\sum_{j=1}^N \Bigl(\sum_{n\in\Bbb{Z}} |D^\sigma\psi_j(n)|^2 \,\Bigr).
\end{eqnarray*}

\end{proof}

\section{Proof of Theorem \ref{L-Tsigma}}

We take the inner product with $\psi_j(n)$ on \eqref{N-Schrodinger} and sum both sides of the equation with respect to $j$. We obtain:
\begin{align*}
\sum_{j=1}^N e_j
&=
\sum_{j=1}^N\Bigl(\sum_{n\in\Bbb{Z}}|D^{\sigma}\psi_j(n)|^2\Bigr) \,\,- \,\,\sum_{j=1}^N\Bigl(\sum_{n\in\Bbb{Z}}b_n|\psi_j(n)|^2\Bigr).
\end{align*}
\noindent We now use Proposition \ref{DGSI} and apply the appropriate H\"older's inequality, i.e:
\begin{equation}
\sum_{j=1}^Ke_j
\geq \sum_{n\in\Bbb{Z}}\Bigl(\sum_{j=1}^N |\psi_j(n)|^2\Bigr)^{2\sigma+1}\, 
-\Bigl(\sum_{n\in\Bbb{Z}} b_n^{(2\sigma+1)/2\sigma}\, \Bigr)^{2\sigma/(2\sigma+1)} 
\Bigl(\sum_{n\in\Bbb{Z}}\Bigl(\sum_{j=1}^N |\psi_j(n)|^2\Bigr)^{2\sigma+1}\, \Bigr)^{1/(2\sigma+1)}.\label{DisSchN}
\end{equation}
\noindent We define 
$$
\chi:= \Bigl(\sum_{n\in\Bbb{Z}}\Bigl(\sum_{j=1}^N |\psi_j(n)|^2\Bigr)^{2\sigma+1}\,\Bigr)^{1/(2\sigma+1)},\,\,\,\quad\,\,\varsigma:=\Bigl(\sum_{n\in\Bbb{Z}} b_n^{(2\sigma+1)/2\sigma}\, \Bigr)^{2\sigma/(2\sigma+1)}.
$$
The latter inequality can be written as
$$
\chi^{2\sigma+1} - \varsigma \chi \le \sum_{j=1}^N e_j \label{dlt1}.
$$ 
The LHS is maximal when
$$
\chi = \frac{1}{(2\sigma+1)^{1/2\sigma}}\Bigl(\sum_{n\in \Bbb{Z}} b_n^{(2\sigma+1)/2\sigma}\,\Bigr)^{1/(2\sigma+1)}.
$$
Substituting this into \eqref{DisSchN}, we obtain:
\begin{eqnarray*}
\sum_{j=1}^N e_j 
& \geq &
\Big(\frac{1}{(2\sigma+1)^{(2\sigma+1)/2\sigma}}\sum_{n\in\Bbb{Z}}         b_n^{(2\sigma+1)/2\sigma}\Big)\,-\frac{1}{(2\sigma+1)^{1/2\sigma}}\sum_{n\in \Bbb{Z}} b_n^{(2\sigma+1)/2\sigma}\,.\\
&   =  &
 \frac{-2\sigma}{(2\sigma+1)^{(2\sigma+1)/2\sigma}}\sum_{n\in\Bbb{Z}} b_n^{(2\sigma+1)/2\sigma}.
\end{eqnarray*}
Therefore:
\begin{equation}\label{L-T1sigma}
\sum_{j=1}^N |e_j|\leq \frac{2\sigma}{(2\sigma+1)^{(2\sigma+1)/2\sigma}}\sum_{n\in\Bbb{Z}} b_n^{(2\sigma+1)/2\sigma}.
\end{equation}

\noindent We lift this bound now with regards to moments by using the standard Aizenman--Lieb procedure. We let $\{e_j(\tau)\}_{j=1}^{N}$ be the negative eigenvalues of the operator $\Delta_D^\sigma -(b_n - \tau)_+$. By the variational principle for the negative eigenvalues $\{-(|e_j| -\tau)_+\}_{j=1}^{N}$ of the operator $\Delta_D^\sigma -(b_n - \tau)$ we have:
$$
(|e_j| - \tau)_+ \leq |e_j(\tau)|.
$$

\noindent By this estimate, we find that
\begin{eqnarray*}
\sum^N_{j=1} |e_j|^\gamma  
&=& 
\frac{1}{\mathcal B(\gamma-1,2)}\, \int_0^\infty\tau^{\gamma-2}(\sum^N_{j=1}|e_j|-\tau)_+ \, d\tau\\
&\leq &
\frac{1}{\mathcal B(\gamma-1,2)},\int_0^\infty \tau^{\gamma-2}\sum^N_{j=1} e_j(\tau)_+\,d\tau\\
&\leq & 
\frac{2\sigma}{(2\sigma+1)^{(2\sigma+1)/2\sigma}}\,\frac{1}{\mathcal B(\gamma-1,2)}\, \int_0^\infty \tau^{\gamma-2}\sum_{n\in\mathbb{Z}} (b_n - \tau)_+^{(2\sigma+1)/2\sigma}d\tau,
\end{eqnarray*}
by \eqref{L-T1sigma} above. Thus, after a change of variable:
\begin{eqnarray*}
\sum^N_{j=1} |e_j|^\gamma 
&\leq & 
\frac{2\sigma}{(2\sigma+1)^{(2\sigma+1)/2\sigma}}\frac{\Gamma(\frac{4\sigma+1}{2\sigma})\Gamma(\gamma+1)}{\Gamma(\gamma+\frac{2\sigma+1}{2\sigma})}\,\sum_{n\in\mathbb{Z}}b_n^{\gamma+1/2\sigma},
\end{eqnarray*}
completing our proof.


\section{Proof of Theorem \ref{JacobiPoly}
}

\noindent We have the following matrix bounds for square, $m\times m$ matrices, as given in \cite{HS02}. For $a_n^m,\, \omega_m\,\in\mathbb{R}$, we have:
\begingroup
\setlength{\arraycolsep}{0.1cm}
\begin{eqnarray*}
\small\left(
        \begin{array}{ccccc}
        -|a^m_n-\omega_m| & 0      &\ldots    & 0       & \omega_m   \\
         0            & 0      & \ldots   & 0      & 0  \\
        \vdots        & \vdots & \ddots   & \vdots & \vdots \\
         0            & 0      & \ldots   & 0      & 0 \\
        \omega_m        & 0      & \ldots   & 0      & -|a^m_n-\omega_m| 
        \end{array} 
\right) 
& \leq & 
 \left(
        \begin{array}{ccccc}
         0 & 0      &\ldots    & 0       & a^m_n   \\
         0            & 0      & \ldots   & 0      & 0  \\
        \vdots        & \vdots & \ddots   & \vdots & \vdots \\
         0            & 0      & \ldots   & 0      & 0 \\
        a^m_n        & 0      & \ldots   & 0      &  0
        \end{array}
\right) 
 \leq 
\left(
        \begin{array}{ccccc}
        |a^m_n-\omega_m| & 0      &\ldots    & 0       & \omega_m   \\
         0            & 0      & \ldots   & 0      & 0  \\
        \vdots        & \vdots & \ddots   & \vdots & \vdots \\
         0            & 0      & \ldots   & 0      & 0 \\
        \omega_m        & 0      & \ldots   & 0      & |a^m_n-\omega_m| 
        \end{array}  
\right).
\end{eqnarray*} 

\endgroup
\noindent  We thus use this on each block of indices of $W_{\sigma}$:
\begin{multline}\label{PolyMatrixBound}
W_\sigma(\{a_n^1\equiv\omega_1\},\ldots,\{a_n^\sigma\equiv\omega_\sigma\},\{b_n^{(-)}\})
\leq 
W_\sigma(\{a_n^1\},\ldots,\{a_n^\sigma\},\{b_n\})\\
\leq 
W_\sigma(\{a_n^1\equiv\omega_1\},\ldots,\{a_n^\sigma\equiv\omega_\sigma\},\{b_n^{(+)}\}).
\end{multline}

\noindent where $b_n^{(\pm)}$ is given by 
\begin{eqnarray*}
b_n^{(\pm)}
&=&
b_n\pm\Big((|a^1_{n-1}-\omega_1|+|a^1_n-\omega_1|)+
 \ldots + (|a^\sigma_{n-\sigma}-\omega_\sigma|+|a^\sigma_{n}-\omega_\sigma|)\Big),
\end{eqnarray*}
i.e.
$$
b_n^{(\pm)}=b_n\pm\Big(\sum^\sigma_{k=1}|a^k_{n-k}-\omega_k|+|a^k_n-\omega_k|\Big).
$$
\noindent Now we relate these to our Schr\"odinger-type operators:
\begin{align}\label{Wsigma+}
\Delta_D^\sigma - 4^\sigma + b_n =& \,\,W_\sigma^0-(4^\sigma-{^{2\sigma}C_{\sigma}})
+ b_n=\,\,W_\sigma\left(\{a^1_n\equiv\omega_1\},\ldots,\{a^\sigma_n\equiv\omega_\sigma\},\{b_n -(4^\sigma-{^{2\sigma}C_{\sigma}}) \}\right), 
\end{align}
\noindent and
\begin{equation}\label{Wsigma-}
\Delta_D^\sigma + b_n = W_\sigma^0+{^{2\sigma}C_{\sigma}}
+b_n=W_\sigma\left(\{a^1_n\equiv\omega_1\},\ldots,\{a^\sigma_n\equiv\omega_\sigma\},\{b_n +{^{2\sigma}C_{\sigma}}\}\right).
\end{equation}
\noindent Now ($E_j^+
-(4^\sigma-{^{2\sigma}C_{\sigma}})$) are positive eigenvalues of $W_\sigma(\{a^1_n\},\ldots,\{a^\sigma_n\},\{b_n -(4^\sigma-{^{2\sigma}C_{\sigma}}) \})$.
Thus by using \eqref{PolyMatrixBound}, and the Variational Principle, we have:
\begin{multline*}
W_\sigma(\{a^1_n\},\ldots,\{a^\sigma_n\},\{b_n -(4^\sigma-{^{2\sigma}C_{\sigma}}) \})
\leq
W_\sigma\left(\{a^1_n\equiv\omega_1\},\ldots,\{a^\sigma_n\equiv\omega_\sigma\},\{b_n^{(+)} -(4^\sigma-{^{2\sigma}C_{\sigma}}) \}\right),
\end{multline*}
\begin{equation}\label{E-e-Bound-sigma}
\Rightarrow|E_j^+-(4^\sigma-{^{2\sigma}C_{\sigma}})|\leq e_j^+,
\end{equation}
where $e_j^+$ are the positive eigenvalues of 
$$
W_\sigma\left(\{a^1_n\equiv\omega_1\},\ldots,\{a^\sigma_n\equiv\omega_\sigma\},\{b_n^{(+)} -(4^\sigma-{^{2\sigma}C_{\sigma}}) \}\right)=\Delta_D^\sigma - 4^\sigma + b_n^{(+)}.
$$
\noindent Let us now define
$
(b_n)_+:=\max (b_n, 0), \,\,  (b_n)_-:= - \min (b_n, 0).
$
Then, by Corollary \ref{LiebGammaN+} for the positive eigenvalues of our operator, we have:
\begin{equation*}
\sum_{j=1}^{N_+}(e^{+}_j)^\gamma \leq \eta^\gamma_\sigma
 \sum_{n\in\mathbb{Z}}  (b_n^{(+)})_{+}^{\gamma+1/2\sigma}.
\end{equation*}

\noindent Thus, applying \eqref{E-e-Bound-sigma}:
\begin{equation}\label{Almost-final1}
\sum_{j=1}^{N_+}|E^{+}_j-(4^\sigma-{^{2\sigma}C_{\sigma}})|^\gamma \leq \eta^\gamma_\sigma
 \sum_{n\in\mathbb{Z}}  \Big((b_n)_{+} +  \sum^\sigma_{k=1}\big(|a^k_{n-k}-\omega_k|+|a^k_n-\omega_k|\big)\Big)^{\gamma+1/2\sigma},
\end{equation}
where
$$
\eta_{\sigma}^{\gamma}:=\frac{2\sigma}{(2\sigma+1)^{(2\sigma+1)/2\sigma}}\frac{\Gamma(\frac{4\sigma+1}{2\sigma})\Gamma(\gamma+1)}{\Gamma(\gamma+\frac{2\sigma+1}{2\sigma})}.
$$

\noindent Similarly, using Theorem \ref{L-Tsigma} on \eqref{Wsigma-}: 
\begin{equation}\label{Almost-final2}
\sum_{j=1}^{N_-}|E^{-}_j+{^{2\sigma}C_{\sigma}}|^\gamma \leq \eta^\gamma_\sigma\sum_{n\in\mathbb{Z}}  \Big((b_n)_{-}+\sum^\sigma_{k=1}\big(|a^k_{n-k}-\omega_k|+|a^k_n-\omega_k|\big)\Big)^{\gamma+1/2\sigma}.
\end{equation}
Using the following application of Jensen's inequality, i.e. for
$\, i\in\{1,\ldots,2\sigma+1\}\,$, let $\alpha_i,\,q \in \mathbb{R},$ with $q\geq1$,
\begin{equation*} 
\left(\sum_{i=1}^{2\sigma+1}\alpha_i\right)^q   \leq (2\sigma+1)^{q-1} \left(\sum_{i=1}^{2\sigma+1}\alpha_i^q\right),
\end{equation*}
to each of \eqref{Almost-final1} and \eqref{Almost-final2}, we have:
\begin{multline*}
\Big((b_n)_{\pm} +  \sum^\sigma_{k=1}\big(|a^k_{n-k}-\omega_k|+|a^k_n-\omega_k|\big)\Big)^{\gamma+1/2\sigma}
\\
\le (2\sigma+1)^{\gamma-(2\sigma-1)/2\sigma}\Big((b_n)_\pm^{\gamma+1/2\sigma}+\sum^\sigma_{k=1}\left(|a^k_{n-k}-\omega_k|^{\gamma+1/2\sigma}+|a^k_n-\omega_k|^{\gamma+1/2\sigma}\right)\Big).
\end{multline*}

\noindent Summing these two inequalities, we arrive at:
\begin{multline*}
\sum_{j=1}^{N_-} |E^{-}_{j}+{^{2\sigma}C_{\sigma}}|^\gamma +
\sum_{j=1}^{N_+} |E^{+}_{j}-(4^\sigma-{^{2\sigma}C_{\sigma}})|^\gamma
\leq
\nu^\gamma_\sigma 
\left(\sum_{n\in\mathbb{Z}} |b_n|^{\gamma+1/2\sigma}+4\sum_{n\in\mathbb{Z}}\sum^\sigma_{k=1}|a^k_{n}-\omega_k|^{\gamma+1/2\sigma}\right),
\end{multline*}
\noindent where
$$
\nu^\gamma_\sigma=2\sigma\,\,(2\sigma+1)^{\gamma-2}\frac{\Gamma(\frac{4\sigma+1}{2\sigma})\Gamma(\gamma+1)}{\Gamma(\gamma+\frac{2\sigma+1}{2\sigma})},
$$

\noindent and the proof of Theorem \ref{JacobiPoly} is complete.

\bibliography{bibliography}{}
\bibliographystyle{alpha}

\end{document}